\documentclass[11pt,reqno,a4paper]{amsart}
 \usepackage{amsfonts}
\usepackage{epsfig}
\usepackage{graphicx}
\usepackage{amsmath}
\usepackage{amssymb}
\usepackage{colordvi}
\usepackage{times,amsmath,epsfig,float,multicol,subfigure}

\usepackage{amsfonts, amssymb, amsmath, amscd, amsthm, txfonts, color, enumerate}

\usepackage{hyperref}

 \usepackage{subfigure}

 \allowdisplaybreaks

\numberwithin{equation}{section}
\numberwithin{figure}{section}
\theoremstyle{plain}
\newtheorem{theorem}{Theorem}[section]

\newtheorem*{conjecturea}{Conjecture A}
\newtheorem*{conjectureb}{Conjecture B}
\newtheorem*{conjecturec}{Conjecture C}

\newtheorem{proposition}[theorem]{Proposition}
\newtheorem{lemma}[theorem]{Lemma}
\newtheorem{corollary}[theorem]{Corollary}

\newtheorem{remark}[theorem]{Remark}
\newtheorem{example}[theorem]{Example}
\newtheorem{definition}[theorem]{Definition}

\title[H\"older continuous maps on the interval  with positive  metric mean dimension]{H\"older continuous maps on the interval   with positive  metric mean dimension}
\author{Jeovanny  M. Acevedo,  Sergio Roma\~na, Raibel Arias}

\address{Jeovanny de Jesus Muentes Acevedo, Facultad de Ciencias B\'asicas,  Universidad Tecnol\'ogica de  Bol\'ivar, Cartagena de Indias - Colombia}
\email{jmuentes@utb.edu.co}
\address{Sergio Augusto Roma\~na Ibarra,   Universidade Federal do Rio de Janeiro,  Rio de Janeiro - Brasil}
\email{sergiori@im.ufrj.br}
\address{Raibel de Jesus Arias Cantillo,    Universidade Federal de Maranh\~ao, Campus Balsas, Brasil}
\email{raibel.jac@ufma.br}

\begin{document}

\begin{abstract}
Fix   a    compact metric space $X$ with finite topological dimension.  Let $C^{0}(X)$  be   the space of continuous maps on $X$ and  $ H^{\alpha}(X)$    the space of $\alpha$-H\"older continuous maps on $X$,   for $\alpha\in (0,1].$ $H^{1}(X)$ is the space of Lipschitz continuous maps on $X$.  We have $$H^{1}(X)\subset H^{\beta}(X) \subset H^{\alpha}(X) \subset C^{0}(X),\quad\text{ where }0<\alpha<\beta<1.$$ 
It is well-known that if $\phi\in H^{1}(X)$, then $\phi$ has metric mean dimension equal to zero. On the other hand, if $X$ is a finite dimensional compact manifold, then $C^{0}(X)$ contains a residual subset whose elements have positive metric mean dimension.  In this work we will prove that, for any $\alpha\in (0,1)$, there exists  $\phi\in H^{\alpha}([0,1]) $ with positive metric mean dimension. 
 \end{abstract}
 
\keywords{metric mean dimension, topological entropy, H\"older continuous maps}

\subjclass[2010]{	54H20, 	37E05, 	37A35}

\date{\today}
\maketitle


\section{Introduction}

 In  \cite{lind},  Lindestrauss and Weiss introduced the notion of  metric mean dimension for any continuous map $\phi:X\rightarrow X$, where $X$ is a compact metric space with metric $d$.    We will   denote by $ 
 \underline{\text{mdim}}_{\text{M}}(X,d,\phi)$ and $\overline{\text{mdim}}_{\text{M}}(X,d,\phi)$, respectively, the lower and upper metric mean dimension of $\phi:X\rightarrow X$.  We have  
 \begin{equation}\label{ejfkfg}
 \underline{\text{mdim}}_{\text{M}}(X,d,\phi)\leq \overline{\text{mdim}}_{\text{M}}(X,d,\phi).\end{equation}
 
  \medskip
  
  Denote by $h_{\text{top}}(\phi)$ the topological entropy of $\phi:X\rightarrow X $. We have if $\underline{\text{mdim}}_{\text{M}}(X,d,\phi)>0$, then $h_{\text{top}}(\phi)=\infty$.
  Example  \ref{equaltozero} proves there exist continuous maps $\phi:X\rightarrow X $ with infinite topological entropy and metric mean dimension equal to zero.  
  
  \medskip
  
  Let $N$ be a compact Riemannian  manifold with topological dimension $\dim(N)=n$.   In  \cite{Carvalho},  the authors proved  if  $n\geq 2$, then the set consisting  of  homeomorphisms on $N$ with upper metric mean dimension equal to $ n$ is residual in $ \text{Hom}(N)$. This fact is proved in  \cite{JeoPMD} for  $C^{0}(N)$ instead of $ \text{Hom}(N)$.    On the other hand, 
  any Lipschitz continuous map defined on a finite dimensional compact metric space has finite entropy (see \cite{Katok}, Theorem 3.2.9), therefore, it has   metric mean dimension equal to zero.

  \medskip
  
  For any $\alpha\in (0,1)$, we denote by $H^{\alpha}([0,1])$ the set consisting of $\alpha$-H\"older continuous maps on the interval $[0,1]$.   Hazard in \cite{Hazard} proves there exist continuous maps on the interval with infinite entropy  which are $\alpha$-H\"older  for any $\alpha\in (0,1)$. However, the example showed by Hazard has zero metric mean dimension (see Example  \ref{equaltozero}). In this work, we will show, for any $\alpha\in (0,1)$, there exists a $\phi\in H^{\alpha}([0,1])$ with   $$ 
 \underline{\text{mdim}}_{\text{M}}([0,1],|\cdot|,\phi)= \overline{\text{mdim}}_{\text{M}}([0,1],|\cdot|,\phi)=1-\alpha.$$
  
  \medskip
  
 In the next section, we will present the definition of metric mean dimension.  In Section \ref{hoorseshoe}, we will show some results about the   metric mean dimension for continuous maps  with horseshoes and several examples. Although the inequality given in \eqref{ejfkfg} is clear,   we do not know   any reference in which is showed an explicit example of a continuous map on the interval  such that the inequality is strict. We will construct this kind of examples in Section \ref{hoorseshoe}. 
  Furthermore, we  prove for   $a,b\in [0,1]$, with $a<b$, the set   consisting of continuous maps $\phi:[0,1]\rightarrow [0,1]$ such that $$      \underline{\text{mdim}}_{\text{M}}([0,1] ,|\cdot |,\phi)=a\quad\text{and}\quad    \overline{\text{mdim}}_{\text{M}}([0,1] ,|\cdot |,\phi)=b$$ is dense in 
   $C^{0}([0,1]).$ In Section \ref{section4} we show the existence of H\"older continuous maps with positive metric mean dimension. Finally, in the last section we will leave some conjectures that arise from this research.  

\medskip

\section{Metric mean dimension for continuous maps}      
Throughout this work,  $X$ will be a compact metric space  endowed with a metric $d$ and $\phi: X\rightarrow X$ a continuous map.       For any  
$n\in\mathbb{N}$, we define the metric $d_n:X\times X\to [0,\infty)$ by
$$
d_n(x,y)=\max\{d(x,y),d(\phi(x),\phi(y)),\dots,d(\phi^{n-1}(x),\phi^{n-1}(y))\}.
$$ \begin{definition}  Fix $\varepsilon>0$. 
\begin{itemize}\item We say that $A\subset X$ is an $(n,\phi,\varepsilon)$-\textit{separated} set
if $d_n(x,y)>\varepsilon$, for any two  distinct points  $x,y\in A$. We denote by $\emph{sep}(n,\phi,\varepsilon)$ the maximal cardinality of any $(n,\phi,\varepsilon)$-{separated}
subset of $X$. 
\item We say that $E\subset X$ is an $(n,\phi,\varepsilon)$-\textit{spanning} set for $X$ if
for any $x\in X$ there exists $y\in E$ such  that $d_n(x,y)<\varepsilon$. Let $\emph{span}(n,\phi,\varepsilon)$ be the minimum cardinality
of any $(n,\phi,\varepsilon)$-spanning subset of $X$.     \end{itemize}\end{definition}
 

 \begin{definition}
  The \emph{topological entropy} of $(X,\phi,d)$   is defined by      
  \begin{equation*}\label{topent}h_{\emph{top}}(\phi)=\lim _{\varepsilon\to0} \emph{sep}(\phi,\varepsilon)=\lim_{\varepsilon\to0}  \emph{span}(\phi,\varepsilon),
\end{equation*}
 where $$\emph{sep}(\phi,\varepsilon)=\underset{n\to\infty}\limsup \frac{1}{n}\log \emph{sep}(n,\phi,\varepsilon)\quad\text{and}\quad\emph{span}(\phi,\varepsilon)=\underset{n\to\infty}\limsup \frac{1}{n}\log \emph{span}(n,\phi,\varepsilon).$$\end{definition}

 \begin{definition}
  We define the \emph{lower  metric mean dimension}   and the \emph{upper metric mean dimension} of $(X,d,\phi )$ by
  \begin{equation*}\label{metric-mean}
 \underline{\emph{mdim}}_{\emph{M}}(X,d,\phi)=\liminf_{\varepsilon\to0} \frac{\emph{sep}(\phi,\varepsilon)}{|\log \varepsilon|}\quad \text{ and }\quad\overline{\emph{mdim}}_{\emph{M}}(X,d,\phi)=\limsup_{\varepsilon\to0} \frac{\emph{sep}(\phi,\varepsilon)}{|\log \varepsilon|},
\end{equation*}
respectively. \end{definition}

We also have that
$$
\underline{\text{mdim}}_{\text{M}}(X,d,\phi )=\liminf_{\varepsilon\to0} \frac{\text{span}(\phi,\varepsilon)}{|\log \varepsilon|} 
\quad\text{ and }\quad \overline{\text{mdim}}_{\text{M}}(X,d,\phi )=\limsup_{\varepsilon\to0} \frac{\text{span}(\phi,\varepsilon)}{|\log \varepsilon|}.
$$

\begin{remark}
Throughout the paper, we will omit the underline and the overline  on the notations of  the metric mean dimension when the result be valid for both cases. 
\end{remark}

\begin{remark}\label{kjjj} Topological entropy does not  depend on the metric $d$. However, the metric mean dimension depends on the metric (see \cite{lind}), therefore, it is not an invariant under topological conjugacy.  
 If $X=[0,1]$, we consider the metric $d(x,y)=|x-y|$, for every $x,y\in [0,1]$. We will denote this metric by $|\cdot|$. 
\end{remark}

\begin{remark} For any  continuous map  $\phi:X\rightarrow X$, it is proved in  \cite{VV} that 
\begin{equation}\label{boundd}0\leq \overline{\emph{mdim}}_{\emph{M}}(X ,d,\phi) \leq \overline{\emph{dim}}_{\emph{B}} (X,d) \quad \text{ and }\quad 0\leq \underline{\emph{mdim}}_{\emph{M}}(X ,d,\phi) \leq \underline{\emph{dim}}_{\emph{B}} (X,d) ,\end{equation}
where $\underline{\emph{dim}}_{\emph{B}} (X,d) $ and $\overline{\emph{dim}}_{\emph{B}} (X,d) $ are respectively the lower and upper box dimension of $X$ with respect to $d$. \end{remark}

From the above remark we have  if $N$ is an $n$-dimensional compact Riemannian manifold with Riemannian metric $d$ we have 
\begin{equation*}\label{dbounddw}  0\leq  {\text{mdim}}_{\text{M}}(N ,d,\phi) \leq n,\end{equation*}
for any continuous map $\phi:N\rightarrow N$. In particular, if $\phi: [0,1]\rightarrow [0,1]$ is a continuous map, we have  
 \begin{equation}\label{bounddw} 0\leq  {\text{mdim}}_{\text{M}}([0,1] ,|\cdot |,\phi) \leq 1.\end{equation}

\section{Horseshoes and metric mean dimension}\label{hoorseshoe}

  An $s$-\textit{horseshoe}   for $\phi:[0,1]\rightarrow [0,1]$ is an interval $J=[a,b]\subseteq [0,1]$ which has a partition into $s$ subintervals $J_{1},\dots,J_{s}$, such that $J_{j}^{\circ}\cap  J_{i}^{\circ}=\emptyset$ for $i\neq j$ and $J\subseteq \phi (\overline{J}_{i})$ for each $i=1,\dots, s$ (in Figure \ref{Ima222} we show a 3-horseshoe).  The subintervals $J_{i}$  will be called \textit{legs} of the horseshoe $J$ and the length $|J|:=b-a$  is its \textit{size}.    
  
  \begin{figure}[ht] 
  \centering
\includegraphics[width=0.3\textwidth]{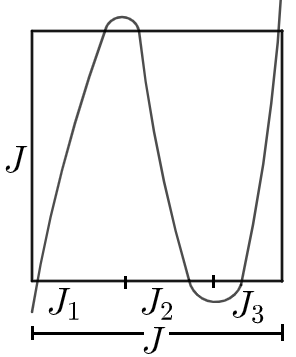}
\caption{$J$ is an $3$-horseshoe}
  \label{Ima222}
\end{figure}
\medskip

Misiurewicz in \cite{Misiurewicz}, proved if $\phi:[0,1]\rightarrow [0,1]$ is a continuous map with $h_{\text{top}}(\phi)>0$,  then there
exist sequences of positive integers $k_{n}$ and $s_{n}$ such that, for each $n$, $\phi^{k_{n}}$ has 
an $s_{n}$- horseshoe and 
$$ h_{\text{top}}(\phi)= \lim_{n\rightarrow \infty}\frac{1}{k_{n}}\log s_{n}.     $$

  For metric mean dimension, in       \cite{VV}, Lemma 6, is proved that: 

\begin{lemma}\label{lema1}   Suppose that $I_{k}=[a_{k-1},a_{k}]\subseteq [0,1]$ is an  $s_{k}$-horseshoe for $\phi:[0,1]\rightarrow [0,1]$ consisting of $s_{k}$ subintervals with the same length $I_{k}^{1},\dots, I_{k}^{s_{k}}$. Setting  $\varepsilon_{k}=\frac{|I_{k}|}{s_{k}}$, 
  we have 
$$ \emph{sep}(  \phi  , \varepsilon_{k})\geq \log(s_{k}/2). $$
\end{lemma}

The above lemma provides a lower bound for the \textit{upper}  metric mean dimension of a continuous map $\phi:[0,1]\rightarrow [0,1]$   with   a sequence of horseshoes, since, with the conditions presented, we can prove that  \begin{equation}\label{feff}\overline{\text{mdim}}_{\text{M}}([0,1],|\cdot|,\phi)\geq \limsup_{k\rightarrow \infty}\frac{\log s_{k}}{|\log\varepsilon_{k}|}= \underset{k\rightarrow \infty}{\limsup}\frac{ 1}{\left|1-\frac{\log |I_{k}|}{\log s_{k}}\right|},   \end{equation} as can be seen  in the proof of   Proposition 8 from \cite{VV}.  In order to obtain the exact value of the (lower and upper) metric mean dimension of a continuous map on the interval, we must be   carefully choosing both the number and the size of the legs of the horseshoes.   Inspired by the examples presented in \cite{Kolyada}, in \cite{JeoPMD}  was proved the next result, which, together with the Lemma \ref{lema1}, will give us examples of continuous maps on the interval with metric mean dimension equal to a fixed value (see Examples   \ref{med1}, \ref{med12} and  \ref{1example}).
 
 \begin{theorem}\label{misiu}
Suppose  for each $k\in\mathbb{N}$ there exists    a   $s_{k}$-horseshoe for $\phi\in C^{0}([0,1])$,  $I_{k}=[a_{k-1},a_{k}]\subseteq [0,1]$, consisting of sub-intervals   with the same length  $I_{k}^{1}, I_{k}^{2},\dots, I_{k}^{s_{k}} $ and $[0,1]=\cup_{k=1}^{\infty}I_{k}$.  
We can rearrange the intervals and suppose that  $2\leq s_{k}\leq s_{k+1}$ for each $k$. If each $\phi|_{I_{k}^{i}}:I_{k}^{i}\rightarrow I_{k} $ is a bijective  affine  map for all $k$ and $i=1,\dots, s_{k}$, we have \begin{enumerate}[i.] 
\item  $   {\underline{\emph{mdim}}_{\emph{M}}}([0,1] ,|\cdot |,\phi) \leq  \underset{k\rightarrow \infty}{\liminf}\frac{1}{\left|1-\frac{\log |I_{k}|}{\log s_{k}}\right|}. $  
\item If the limit $\underset{k\rightarrow \infty}{\lim}\frac{1}{\left|1-\frac{\log |I_{k}|}{\log s_{k}}\right|}$ exists, then $\overline{\emph{mdim}}_{\emph{M}}([0,1] ,|\cdot |,\phi) =  \underset{k\rightarrow \infty}{\lim}\frac{1}{\left|1-\frac{\log |I_{k}|}{\log s_{k}}\right|}.$ \end{enumerate}
 \end{theorem}

  In Figure \ref{Ima22} we present the graph of a  continuous map $\phi:[0,1]\rightarrow [0,1]$ such that each $I_{k}$ is an  $3^{k}$-horseshoe for $\phi$.  Note that in   Theorem \ref{misiu},i.  is presented an upper bound for the \textit{lower} metric mean dimension and in  ii. a condition   to obtain   the exact  value of the \textit{upper} metric mean dimension for a certain class of continuous maps on the interval. 
\begin{figure}[ht] 
  \centering
\subfigure[Each $I_{k}$ is an $3^{k}$-horseshoe]{\includegraphics[width=0.34\textwidth]{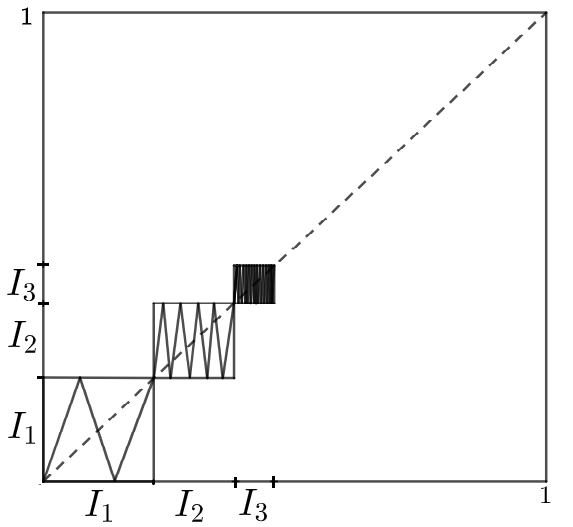} \label{Ima22}}\quad \subfigure[Each $I_{k}$ is an $3^{k}$-horseshoe]{\includegraphics[width=0.34\textwidth]{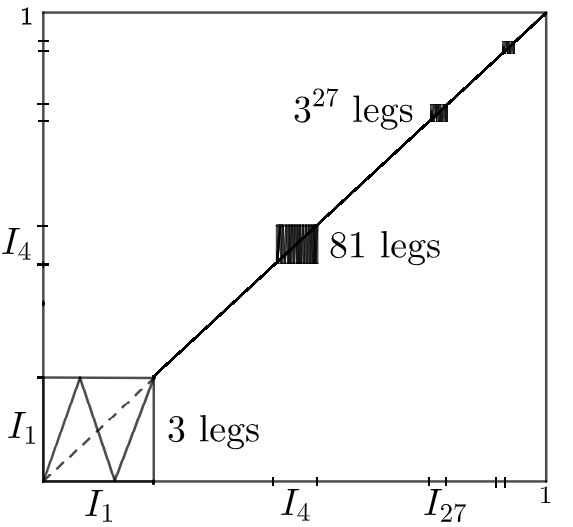}\label{Ima223}}
\caption{$s$-horseshoes}
\end{figure}
 
In the next proposition we show a lower bound for the \textit{lower} metric mean dimension of a continuous map on the interval satisfying the conditions in Theorem \ref{misiu}.
 
\begin{proposition}\label{ddf}
If $\phi:[0,1]\rightarrow [0,1]$ satisfies the conditions in Theorem \ref{misiu}, we have \begin{equation}\label{fefeesc}   {\underline{\emph{mdim}}_{\emph{M}}}([0,1] ,|\cdot |,\phi) \geq \underset{k\rightarrow \infty}{\liminf}\frac{\frac{\log s_{k-1}}{\log s_{k}}}{1 - \frac{\log |I_{k}|}{\log s_{k}}} . \end{equation}
If $\underset{k\rightarrow \infty}{\liminf}\frac{\frac{\log s_{k-1}}{\log s_{k}}}{1 - \frac{\log |I_{k}|}{\log s_{k}}}=  \underset{k\rightarrow \infty}{\liminf}\frac{1}{\left|1-\frac{\log |I_{k}|}{\log s_{k}}\right|},$ we have $$   {\underline{\emph{mdim}}_{\emph{M}}}([0,1] ,|\cdot |,\phi) =  \underset{k\rightarrow \infty}{\liminf}\frac{1}{\left|1 - \frac{\log |I_{k}|}{\log{ s_{k}}}\right|} . $$
\end{proposition} 
 
\begin{proof}
  Take any $\varepsilon\in (0,1)$. For any $k\geq 1$ set $ \varepsilon_k=  \frac{|I_{k}|}{s_{k}}  .$   There exists  some $k\geq 1$ such that  $\varepsilon \in [\varepsilon_{k}, \varepsilon_{k-1}]$. We have  
 \begin{equation*}\label{equ12ssswe}   \text{sep}( n, \phi  , \varepsilon)\geq \text{sep}( n, \phi  , \varepsilon_{k-1})\geq   \text{sep}( n, \phi  |_{I_{k-1}}, \varepsilon_{k-1}), \quad\text{for any }n\geq 1,  \end{equation*}
 and thus $ \text{sep}(  \phi  , \varepsilon)\geq     \text{sep}(  \phi  |_{I_{k-1}}, \varepsilon_{k-1}). $
Hence, from Lemma \ref{lema1}, it follows that  
$     
  \text{sep}(\phi, \varepsilon)\geq    \log \left(\frac{s_{k-1}}{2}\right) .$
 Therefore,  \begin{align*}\label{exxample12} {\underline{\text{mdim}}_{\text{M}}}([0,1] ,|\cdot |,\phi) &= \liminf_{\varepsilon\rightarrow 0} \frac{\text{sep}(\phi  ,\varepsilon )}{|\log  {\varepsilon}|} \geq \liminf_{k\rightarrow \infty}\frac{ \log {s_{k-1}}}{|\log \varepsilon_{k}|}\\
 &=\underset{k\rightarrow \infty}{\liminf}\frac{\log s_{k-1}}{\log s_{k} - \log |I_{k}|}=\underset{k\rightarrow \infty}{\liminf}\frac{\frac{\log s_{k-1}}{\log s_{k}}}{1 - \frac{\log |I_{k}|}{\log s_{k}}}.\end{align*} 
 The second part of the proposition follows from the above fact and Theorem \ref{misiu}. 
\end{proof}

As we mentioned above, Lemma  \ref{lema1} provides a lower bound for the upper metric mean dimension of a continuous map on the interval. For instances,  in  \cite{VV}, Proposition 8, is proved if  $\varrho:[0,1]\rightarrow [0,1]$ satisfies the conditions in Lemma  \ref{lema1}, with $s_{k}=k^{k}$ and $|I_{k}|=\frac{6}{\pi^{2} k^{2}}$ for any $k\in \mathbb{N}$, then  \begin{equation}\label{mapavv}\overline{\text{mdim}}_{\text{M}}([0,1] ,|\cdot |,\varrho)=1.\end{equation} This fact is a consequence of \eqref{feff} and \eqref{bounddw}. Next, it  follows from \eqref{fefeesc} that \begin{equation}\label{defeeee} {\underline{\text{mdim}}_{\text{M}}}([0,1] ,|\cdot |,\varrho) \geq \underset{k\rightarrow \infty}{\liminf}\frac{\frac{\log s_{k-1}}{\log s_{k}}}{1 - \frac{\log |I_{k}|}{\log s_{k}}}=\underset{k\rightarrow \infty}{\liminf}\frac{\frac{\log (k-1)^{k-1}}{\log k^{k}}}{1 - \frac{\log |6/\pi^{2}k^{2}|}{\log k^{k}}}=1.\end{equation} Therefore, from  \eqref{mapavv} and \eqref{defeeee} we have $$\overline{\text{mdim}}_{\text{M}}([0,1] ,|\cdot |,\varrho)=\underline{\text{mdim}}_{\text{M}}([0,1] ,|\cdot |,\varrho)=1.$$

We will obtain a  continuous map $\varphi_{0,1}:[0,1]\rightarrow [0,1]$ such that $$   0=   \underline{\text{mdim}}_{\text{M}}([0,1] ,|\cdot |,\varphi_{0,1})<   \overline{\text{mdim}}_{\text{M}}([0,1] ,|\cdot |,\varphi_{0,1})=1.$$

  \begin{example}\label{med1}   Set $a_{0}=0$ and $a_{n}= \sum_{i=1}^{n}\frac{6}{\pi^{2}i^{2}}$ for $n\geq 1$. Set  $I_{n}:=[a_{n-1},a_{n}]$ for any $n\geq 1$.    Let $\varphi_{0,1}\in C^{0}([0,1])$ be defined by \begin{equation*}  \varphi_{0,1} (x) =\begin{cases}
    T_{n^{n}}^{-1}\circ g^{n^{n}}\circ T_{n^{n}} & \text{if }x\in I_{n^{n}},\text{ for }n\geq 1\\
    x & \text{ otherwise} 
  \end{cases},\end{equation*} where   $T_{n}: I_{n}:=[a_{n},a_{n+1}] \rightarrow [0,1] $ is the unique increasing affine map from $I_{n}$   onto $[0,1]$ and   $g(x)= |1-|3x-1||$ for any $x\in[0,1]$.  Each $I_{n^{n}}$ is an     $3^{n^{{n}}}$-horseshoe for $\varphi_{0,1}$ (see Figure \ref{Ima223}).  It  follows from    Theorem \ref{misiu} that  $$ {\overline{\emph{mdim}}_{\emph{M}}}([0,1] ,|\cdot |,\varphi_{0,1}) =   \underset{n\rightarrow \infty}{\lim}\frac{1}{\left|1-\frac{\log |I_{n^{n}}|}{\log 3^{n^{n}}}\right|} =   \underset{n\rightarrow \infty}{\lim}\frac{1}{\left|1-\frac{\log \left|\frac{6}{\pi^{2}(n^{n})^{2}}\right|}{\log 3^{n^{n}}}\right|}=1. $$
  Next, we prove that $$ {\underline{\emph{mdim}}_{\emph{M}}}([0,1] ,|\cdot |,\varphi_{0,1}) =  0. $$ 
  Firstly, given that $\varphi_{0,1}$ is the identity outside of  $Y=\cup_{j=1}^{\infty}I_{j^{j}}$, we have 
  $$  {\underline{\emph{mdim}}_{\emph{M}}}([0,1] ,|\cdot |,\varphi_{0,1})= {\underline{\emph{mdim}}_{\emph{M}}}(Y ,|\cdot |,\varphi_{0,1}|_{Y}). $$
  Next, note that $\frac{\log  3^{(k-1)^{k-1}}}{\log 3^{k^{k}}}\rightarrow 0$ as $k\rightarrow \infty$. Hence, for any $\delta >0$ there exists $k_{0}\geq 1$ such that for any  $k> k_{0}$ we have $\frac{\log  3^{(k-1)^{k-1}}}{\log 3^{k^{k}}}< \delta$. For any $k\geq k_{0}$, set $ \varepsilon_k=  \frac{|I_{k^{k}}|}{3(3^{k^{k}})}  =\frac{6}{3^{k^{k}+1}\pi^{2}k^{2k}} .$  It follows from Corollary 7.2, \cite{demelo}, page 165, for each $j=1, 2,\dots, {k-1}$,   we have   $$ \emph{span}(n,\varphi_{0,1} |_{I_{j^{j}}}, {\varepsilon_{k}})\leq \frac{(3^{j^{j}})^{n}}{\varepsilon_{k}}.$$   Hence, if $Y_{k}=\cup_{j=1}^{k-1}I_{j^{j}}$, for every $ n\geq1$ we have 
  \begin{align*}    \emph{span}(n,\varphi_{0,1} |_{Y_{k}}, {\varepsilon_{k}}) & \leq \sum_{j=1}^{k-1}   \frac{(3^{j^{j}})^{n}}{\varepsilon_{k}}  \leq   \sum_{j=1}^{k-1}   \frac{(3^{(k-1)^{k-1}})^{n}}{\varepsilon_{k}}    \leq  (k-1)   \frac{(3^{(k-1)^{k-1}})^{n}}{\varepsilon_{k}}.  \end{align*}
 Therefore,  
 \begin{align*} \limsup_{n\rightarrow \infty}\frac{\emph{span}(\varphi_{0,1} |_{Y_{k}}  , {\varepsilon_{k}})}{n|\log {\varepsilon_{k}}|}&\leq \limsup_{n\rightarrow \infty}\frac{\log\left[(k-1)   \frac{(3^{(k-1)^{k-1}})^{n}}{\varepsilon_{k}}\right]}{n[\log (3^{k^{k}+1}\pi^{2}k^{2k}/ 6)]} \leq  \frac{\log  3^{(k-1)^{k-1}}}{\log 3^{k^{k}}}.
 \end{align*}
 This fact implies that for any $\delta >0$ we have $$  {\underline{\emph{mdim}}_{\emph{M}}}(Y ,|\cdot |,\varphi_{0,1}|_{Y})<\delta$$ and hence $$ {\underline{\emph{mdim}}_{\emph{M}}}([0,1] ,|\cdot |,\varphi_{0,1})={\underline{\emph{mdim}}_{\emph{M}}}(Y ,|\cdot |,\varphi_{0,1}|_{Y})=0.$$
  \end{example}
 
Fix $a\in [0,1]$ and let $\phi_{a}\in C^{0}([0,1])$ be such that $$ {\underline{\text{mdim}}_{\text{M}}}([0,1] ,|\cdot |,\phi_{a})={\overline{\text{mdim}}_{\text{M}}}([0,1] ,|\cdot |,\phi_{a})=a$$ (two different constructions of this kind of maps can be seen  in  \cite[Example 3.1]{JeoPMD} and \cite[Proposition 9.1]{Carvalho}).  Set $ {\varphi}_{a,1}\in C^{0}([0,1])$ be defined by \begin{equation*}   {\varphi}_{a,1} (x) =\begin{cases}
    T_{1}^{-1}\circ \varphi_{0,1}\circ T_{1} & \text{if }x\in [0,\frac{1}{2}],\\
     T_{2}^{-1}\circ \phi_{a}\circ T_{2} & \text{if }x\in [\frac{1}{2},1]  
  \end{cases},\end{equation*} where   $T_{1}:  [1,\frac{1}{2}] \rightarrow [0,1] $ and $T_{2}:  [\frac{1}{2},1] \rightarrow [0,1] $  are, respectively, the unique increasing affine map from $[0,\frac{1}{2}]$   onto $[0,1]$ and   from $[\frac{1}{2},1]$   onto $[0,1]$.
 We have \begin{equation}\label{a1} {\underline{\text{mdim}}_{\text{M}}}([0,1] ,|\cdot |, {\varphi}_{a,1})=a\quad\text{and}\quad {\overline{\text{mdim}}_{\text{M}}}([0,1] ,|\cdot |,\varphi_{a,1})=1.\end{equation}
 
 Next, for fixed $b \in (0,1)$, we present an example of a continuous map $\varphi_{0,b}\in C^{0}([0,1])$ such that \begin{equation*}\label{a1e2}  {\underline{\text{mdim}}_{\text{M}}}([0,1] ,|\cdot |, {\varphi}_{0,b})=0\quad\text{and}\quad {\overline{\text{mdim}}_{\text{M}}}([0,1] ,|\cdot |,\varphi_{0,b})=b.\end{equation*}
 
  \begin{example}\label{med12}  Fix  $r>0$ and set $ b=\frac{1}{r+1}$.  Set $a_{0}=0$ and $a_{n}= \sum_{i=0}^{n-1}\frac{C}{3^{ir}}$ for $n\geq 1$, where $C=\frac{1}{\sum_{i=0}^{\infty}\frac{1}{3^{ir}}}= \frac{3^{r}-1}{3^{r}}$.     Let $\varphi_{0,b}\in C^{0}([0,1])$ be defined by \begin{equation*}  \varphi_{0,b} (x) =\begin{cases}
    T_{n^{n}}^{-1}\circ g^{n^{n}}\circ T_{n^{n}} & \text{if }x\in I_{n^{n}},\text{ for }n\geq 1\\
    x & \text{ otherwise} 
  \end{cases},\end{equation*} where   $T_{n}: I_{n}:=[a_{n-1},a_{n}] \rightarrow [0,1] $ is the unique increasing affine map from $I_{n}$   onto $[0,1]$ and   $g(x)= |1-|3x-1||$ for any $x\in[0,1]$.  Each $I_{n^{n}}$ is a     $3^{n^{{n}}}$-horseshoe for $\varphi_{0,b}$.  It  follows from    Theorem \ref{misiu} that  $$ {\overline{\emph{mdim}}_{\emph{M}}}([0,1] ,|\cdot |,\varphi_{0,b}) =   \underset{n\rightarrow \infty}{\lim}\frac{1}{\left|1-\frac{\log |I_{n^{n}}|}{\log 3^{n^{n}}}\right|} =   \underset{n\rightarrow \infty}{\lim}\frac{1}{\left|1+\frac{\log (3^{n^{n}})^{r}}{\log 3^{n^{n}}}\right|}=\frac{1}{1+r}=b. $$
  We can prove that $$ {\underline{\emph{mdim}}_{\emph{M}}}([0,1] ,|\cdot |,\varphi_{0,b}) =  0$$  analogously as in Example \ref{med1}.
  \end{example}
  
  For $a<b$, let $\phi_{a}\in C^{0}([0,1])$ be such that $$ {\underline{\text{mdim}}_{\text{M}}}([0,1] ,|\cdot |,\phi_{a})={\overline{\text{mdim}}_{\text{M}}}([0,1] ,|\cdot |,\phi_{a})=a.$$   Set $ {\varphi}_{a,b}\in C^{0}([0,1])$ be defined by \begin{equation*}   {\varphi}_{a,b} (x) =\begin{cases}
    T_{1}^{-1}\circ \varphi_{0,b}\circ T_{1} & \text{if }x\in [0,\frac{1}{2}],\\
     T_{2}^{-1}\circ \phi_{a}\circ T_{2} & \text{if }x\in [\frac{1}{2},1]  
  \end{cases}.\end{equation*} 
 We have \begin{equation}\label{ab1} {\underline{\text{mdim}}_{\text{M}}}([0,1] ,|\cdot |, {\varphi}_{a,b})=a\quad\text{and}\quad {\overline{\text{mdim}}_{\text{M}}}([0,1] ,|\cdot |,\varphi_{a,b})=b.\end{equation}

We will endow  $C^{0}([0,1])$   with the metric 
   $$\hat{d}(\phi,\psi)=\max_{x\in [0,1]}|\phi(x)-\psi(x)|.
$$  
In \cite{VV}, Proposition 9, is proved the set consisting of $\phi\in C^{0}([0,1])$ with $\overline{\text{mdim}}_{\text{M}}([0,1],|\cdot|,\phi)=1$ is dense in $  C^{0}([0,1])$.  More generally, in \cite{Carvalho}, Theorem B, is proved for a fixed $a\in [0,1]$, the set consisting of continuous maps $\phi\in C^{0}([0,1])$ such that $$      \underline{\text{mdim}}_{\text{M}}([0,1] ,|\cdot |,\phi)= \overline{\text{mdim}}_{\text{M}}([0,1] ,|\cdot |,\phi)=a$$ is dense in 
   $C^{0}([0,1])$ (see also \cite{JeoPMD}, Theorem 4.1).    Furthermore, we have: 
  \begin{theorem}
  For   $a,b\in [0,1]$, with $a<b$, the set $C_{a}^{b}([0,1])$ consisting of continuous maps $\phi:[0,1]\rightarrow [0,1]$ such that $$      \underline{\emph{mdim}}_{\emph{M}}([0,1] ,|\cdot |,\phi)=a\quad\text{and}\quad    \overline{\emph{mdim}}_{\emph{M}}([0,1] ,|\cdot |,\phi)=b$$ is dense in 
   $C^{0}([0,1]).$
  \end{theorem}
  \begin{proof}  From \eqref{a1} and  \eqref{ab1} we have for any $a,b\in[0,1]$, with $a\leq b$, there exists $\varphi_{a,b}\in C_{a}^{b}([0,1])$. 
 Fix $\psi\in C^{0}([0,1])$ and take any $\varepsilon >0$. Given that $C^{1}([0,1])$ is dense in $C^{0}([0,1])$, we can assume that $\psi\in C^{1}([0,1])$ and therefore $$      \underline{\text{mdim}}_{\text{M}}([0,1] ,|\cdot |,\phi)=     \overline{\text{mdim}}_{\text{M}}([0,1] ,|\cdot |,\phi)=0.$$  
 Let $p^{\ast}$ be a fixed point of $\psi$.   Choose  $\delta>0$ such that $|\psi(x)-\psi (p^{\ast})|<\varepsilon/2$ for any $x$ with $|x-p^{\ast}|<\delta$.  Take $\varphi_{a,b}\in C_{a}^{b}([0,1])$. Set  $J_{1}=[0,p^{\ast}]$, $J_{2}=[p^{\ast},p^{\ast}+\delta/2]$, $J_{3}=[p^{\ast}+\delta/2,p^{\ast}+\delta] $ and $J_{4}=[p^{\ast}+\delta,1]$. Take the continuous map $\psi_{a,b}:[0,1]\rightarrow [0,1]$ defined as $$
\psi_{a.b}(x)= \begin{cases}
    \psi(x), &  \text{ if }x\in J_{1}\cup J_{4}, \\
   T_{2}^{-1}\varphi_{a,b}T_{2}(x), &  \text{ if }x\in J_{2}, \\
    T_{3}(x), &  \text{ if }x\in J_{3}, 
      \end{cases}
$$ where    $T_{2}:J_{2}\rightarrow [0,1] $  is the  affine map such that $T_{2}(p^{\ast})=0 $ and $T_{2}(p^{\ast}+\delta/2)=1 $,   and  $ T_{3}:J_{3}\rightarrow [ p^{\ast}+\delta/2, \psi(p^{\ast}+\delta)]$ is the  affine map   such that $ T_{3}(p^{\ast}+\delta/2)=p^{\ast}+\delta/2$  and $ T_{3} (p^{\ast}+\delta)=\psi(p^{\ast}+\delta)$.  Note that $\hat{d}(\psi_{a,b},\psi)<\varepsilon.$ We have  $$\underline{\text{mdim}}_\text{M}([0,1] ,|\cdot |,\psi_{a,b})
= \underline{\text{mdim}}_\text{M}(J_{2} ,|\cdot |,\varphi_{a,b}) 
=a$$
and $$ \overline{\text{mdim}}_\text{M}([0,1] ,|\cdot |,\psi_{a,b})
= \overline{\text{mdim}}_\text{M}(J_{2} ,|\cdot |,\varphi_{a,b}) 
=b,$$ 
which proves the theorem.
 \end{proof}

 \section{H\"older continuous maps with positive metric mean dimension}\label{section4}
 
   We say that $\phi:[0,1]\rightarrow [0,1]$ is an $\alpha$-\textit{H\"older continuous map}, for $\alpha\in (0,1]$, if there exists $K>0$ such that
$$ \frac{|\phi(x)-\phi(y)|}{|x-y|^{\alpha}}\leq K\quad \text{for all }x, y\in [0,1],\text{ with }x\neq y.    $$
If $\phi$ satisfies the above condition for  $\alpha =1$, then $\phi $ is called a \textit{Lipschitz continuous map}.
 
\medskip

 For $\alpha\in (0,1),$   $ H^{\alpha}([0,1])$  will denote  the space of $\alpha$-H\"older continuous maps on $[0,1]$. $C^{1}([0,1])$ and   $H^{1}([0,1])$ will   denote respectively the space of $C^{1}$-maps and the space of Lipschitz continuous maps on $[0,1]$.  We have $$C^{1}([0,1])\subset H^{1}([0,1])\subset H^{\beta}([0,1]) \subset H^{\alpha}([0,1]) \subset C^{0}([0,1]),\quad\text{ where }0<\alpha<\beta<1.$$

  
 Next, suppose that $N$ is a compact Riemannian manifold with topological dimension equal to $n$ and  Riemannian metric $d$. In \cite{JeoPMD}, Theorem 4.5,  is proved  the set consisting of continuous maps on $N$ with \textit{lower} and \textit{upper} metric mean dimension equal to a fixed $a\in [0,n]$ is dense in $C^{0}(N)$. Furthermore, in   Theorem 4.6, is showed the  set consisting of continuous maps on $N$ with \textit{upper} metric mean dimension equal to $n$ is residual in $C^{0}(N)$. If $n\geq 2$,  in \cite{Carvalho}, Theorem A, is proved  the set consisting of homeomorphisms on $N$ with \textit{upper} metric mean dimension equal to $n$ is residual in the set consisting of homeomorphisms on $N$. It is well known any homeomorphism on $[0,1]$ has zero entropy and therefore has zero metric mean dimension. 
 
 \medskip
 
 Hence, it remains to show the existence of   H\"older continuous maps on finite dimensional compact manifolds. In Theorem \ref{example-a}  we will prove   there exist H\"older continuous maps on the interval  with positive metric mean dimension, with certain conditions of the H\"older exponent.   In Conjectures A, B and C the authors leave three problems that can be objects of future studies.

 \medskip

The next lemma, whose proof is straightforward and
left to the reader, will be useful in order to prove that some functions are H\"older continuous.

\begin{lemma}\label{hfnff} Let $\phi:[0,1]\rightarrow [0,1]$ be a continuous map such that $ \frac{|\phi(x)-\phi(y)|}{|\omega(x-y)|}\leq K$ for some $K>0$ and any $x\neq y\in [0,1]$, where $\omega(t)=-t\log(t)$  for any $ t>0$ and $\omega (t)=0$. Then $\phi$ is an  $\alpha$-H\"older continuous map for any $\alpha\in (0,1)$. 
\end{lemma}
 
 \begin{definition} Let $\phi:[0,1]\rightarrow [0,1]$ be a continuous map. If there exists $K>0$ such that  $ \frac{|\phi(x)-\phi(y)|}{|\omega(x-y)|}\leq K$ for any $x\neq y\in [0,1]$,  we will say that $\phi$ has \textit{modulus of continuity} $\omega$. 
\end{definition}
 
     Next example, which was  introduced in \cite{Hazard}, proves there exist continuous maps  with infinite entropy    and  metric mean dimension equal to zero. That map is also   $\alpha$-H\"older  for any $\alpha\in (0,1)$. We will  include the details of its construction for the sake of completeness. 

\begin{example}\label{equaltozero} For $n\geq 1$, take   $I_{n}=[2^{-n},2^{-n+1}]$. Note that $$|I_n|=2^{-n+1}-2^{-n}=2^{-n}(2-1)=2^{-n} \quad\text{for each }n.$$ Divide each interval $I_{n}$ into $2n+1$ sub-intervals with the same lenght, $I_{n}^{1}$, $\dots$, $I_{n}^{2n+1}$. For $k=1,3,\dots, 2n+1$, let $\phi|_{I_{n}^{k}} :I_{n} ^{k}\rightarrow I_{n} $ be  the unique increasing affine map from $I_{n} ^{k}$ onto $I_{n}$ and for $k=2,4,\dots, 2{n}$, let $\phi|_{I_{n}^{k}} :I_{n} ^{k}\rightarrow I_{n} $ be  the unique decreasing affine map from $I_{n} ^{k}$ onto $I_{n}$.  Note that $\phi$ is a continuous map ($\phi(x)=x$ for any $x\in \partial I_{n}$) and  each $I_{n}$ is a $(2n+1)$-horseshoe  (see Figure \ref{Ima1}), therefore $h_{\emph{top}}(\phi)=\infty$. It follows from Theorem \ref{misiu} that $$\emph{mdim}_{\emph{M}}([0,1] ,| \cdot |,\phi)=0.$$  We will prove 
  $\phi$ is $\alpha$-H\"older for any $\alpha\in(0,1).$ We will consider the map  $\omega $   defined in Lemma \ref{hfnff}. \begin{figure}[ht] 
  \centering
\includegraphics[width=0.28\textwidth]{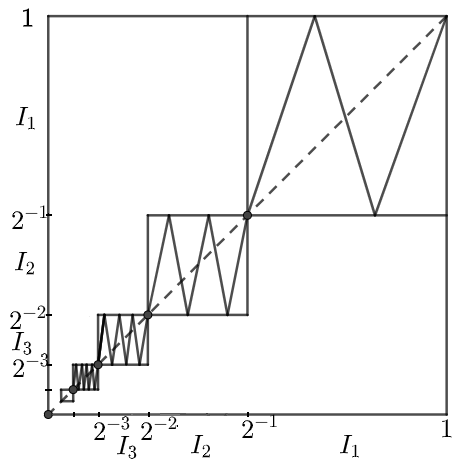}
\caption{Each $ I_{n}$ is a $(2n+1)$-horseshoe}
  \label{Ima1}
\end{figure}
Let $x,y\in [0,1]$ be two distinct points.  Then there exist $n\geq 1$ and $m\geq 1$ such that $x\in I_{n}$ and $y\in I_{m}$. Given that every $I_{k}$ is $\phi$-invariant, we have $\phi(x)\in I_{n}$ and $\phi(y)\in I_{m}$. We have the next cases:

\medskip

\noindent \textbf{Case} $x=0:$ We have $\phi(0)=0$ and hence $|\phi(0)-\phi(y)|=|\phi(y)|\leq 2^{-m+1}$. Furthermore,  $$  2^{-m}\leq|y|\leq2^{-m+1} \quad\text{ and thus }\quad \frac{1}{|y|}\leq 2^{m}\quad\text{and}\quad \frac{1}{\log\frac{1}{|y|}}\leq \frac{1}{\log 2^{m-1} } . $$ Hence, $$ \frac{ |\phi(y)|}{\omega(|y|)}<\frac{2^{-m+1}2^{m}}{\log 2^{m-1}} =\frac{2}{(m-1)\log 2}.$$

\noindent \textbf{Case} $n> m+1:$  In this case we have 
$$ |\phi(x)-\phi(y)|\leq |2^{-n}-2^{-m+1}|< 2^{-m+1}   \quad\text{and} \quad 2^{-m+1}>|x-y|>|2^{-n+1}-2^{-m}|>2^{-m} .$$ 
Therefore, $$ \frac{1}{|x-y|}<2^{m}\quad\text{and}\quad \frac{1}{\log\frac{1}{|x-y|}}< \frac{1}{\log 2^{m-1} }  $$ and thus
$$ \frac{ |\phi(x)-\phi(y)|}{\omega(|x-y|)}<\frac{2^{-m+1}2^{m}}{\log 2^{m-1}} =\frac{2}{(m-1)\log 2}.$$ 

 \noindent  \textbf{Case} $n=m+1:$ Note that $I_m=[2^{-m}, 2^{-m+1}],$ $ I_{m+1}=[2^{-(m+1)}, 2^{-m}]$ and consider the sub-intervals \begin{align*}
I^1_{m}&=\biggl[2^{-m}, 2^{-m}+\frac{|I_{m}|}{2m+1}\biggr]=\biggl[2^{-m}, 2^{-m}+\frac{2^{-m}}{2m+1}\biggr]=\biggl[2^{-m}, \frac{2m+2}{2^m(2m+1)}\biggr]=[B_m, C_m]\subseteq I_{m}
\end{align*} and 
\begin{align*}
I^{2m+1}_{m+1} &=\biggl[2^{-m}-\frac{1}{2^{m+1}(2m+3)},2^{-m}\biggr]=\biggl[\frac{4m+5}{2^{m+1}(2m+3)}, 2^{-m}\biggr]=[A_m, B_m]\subseteq I_{m+1}.
\end{align*}
For any $x\in I^{2m+1}_{m+1}$, we have 
\begin{align*}
\phi  (x)&= \frac{2^{-m}-2^{-(m+1)}}{2^{-m}-A_m}\ (x-2^{-m})+2^{-m}=\frac{2^{-m}-2^{-(m+1)}}{2^{-m}-\frac{4m+5}{2^{m+1}(2m+3)}}\ (x-2^{-m})+2^{-m}\\
&=(2m+3)(x-2^{-m})+2^{-m}.
\end{align*}
For any $x\in I^{1}_{m}$, we have
\begin{align*}
\phi (x)&= \frac{2^{-m+1}-2^{-m}}{C_m-2^{-m}} \ (x-2^{-m})+2^{-m}=(2m+1)(x-2^{-m})+2^{-m}.
\end{align*}
Hence, if $x\in I^{2m+1}_{m+1}$ and $y\in I^{1}_{m}$, we have \begin{align*}
|x-y|\leq C_m-A_m &=\frac{2m+2}{2^{m}(2m+1)}-\frac{4m+5}{2^{m+1}(2m+3)}\\
& =\frac{1}{2^m}\biggl[\frac{2m+1}{2m+1}+\frac{1}{2m+1}-\frac{2(2m+3)}{2(2m+3)}+\frac{1}{2(2m+3)}\biggr]\\
&=\frac{1}{2^m}\biggl[\frac{1}{2m+1}+\frac{1}{2(2m+3)}\biggr]\\
&\leq \frac{1}{2^m}\ \frac{2}{2m+1}=\frac{1}{2^{m-1}(2m+1)}
\end{align*} and, furthermore,
$$
|\phi(y)-\phi(x)|= (2m+1)(y-2^{-m})-(2m+3)(x-2^{-m})\leq (2m+3)(y-x).$$
Therefore, 
\begin{align*}
\frac{|\phi(y)-\phi(x)|}{\omega(|x-y|)}& \leq \frac{(2m+3)(y-x)}{-|x-y| \log |x-y|}\leq \frac{2m+3}{\log (2^{m-1}(2m+1))}.
\end{align*}
Given that 

$$\lim_{m\to \infty}\frac{2m+3}{(m-1)\log 2+ \log (2m+1) }=\frac{2}{\log 2},$$
we have 
\begin{equation}\label{escfr}
\frac{|\phi(y)-\phi(x)|}{\omega(|x-y|)}\leq 3.
\end{equation} 

If $x\in I^{k}_{m+1}$ and $y\in I^{j}_{m}$, then there are ${x}'\in I^{2m+1}_{m+1}$ and ${y}'\in I^{1}_{m}$ such that $\phi(x)= \phi({x}')$ and $\phi(y)= \phi({y}')$.   We have 
$$
 |{x}'-{y}'|\leq |x-y|\leq C_m-A_m= \frac{1}{2^m}\ \frac{2}{2m+1}=\frac{1}{2^{m-1}(2m+1)}<e^{-1}, \quad\text{for all } \ m\geq 2.
$$ Since $\omega(t)$ is increasing on $[0, e^{-1}]$, it  follows that
 $ \omega (|{x}'-{y}'|)\leq \omega (|x-y|)$ and therefore, by \eqref{escfr}, we have
$$  \frac{|\phi(y)-\phi(x)|}{\omega(|y-x|)}\leq  \frac{|\phi({y}')-\phi({x}')|}{\omega(|{y}'-{x}'|)}\leq 3.$$
 Note also, if $m=1$, then $\phi|_{I_1\cup I_2}$ is a Lipschitz function. Hence,  in particular, it has modulus of  continuity $\omega $.  
 
\medskip

\noindent \textbf{Case} $n=m$: In this case we have there exists $y^{\prime}$ in the same branch as $x$ such that $f(y^{\prime})=f(y)$ and $|x-y|\geq |x-y^{\prime}|$. Note that  $$ |x-y^{\prime}|\leq \frac{|I_{n}|}{2n+1}=\frac{1}{2^{n}(2n+1)} \quad\text{and hence }\quad \frac{1}{\log(|x-y^{\prime}|^{-1})}\leq \frac{1}{\log(2^{n}(2n+1))} .  $$ Therefore, \begin{align*}\frac{ |\phi(x)-\phi(y)|}{\omega(|x-y|)}& \leq \frac{ |\phi(x)-\phi(y^{\prime})|}{\omega(|x-y^{\prime}|)} = \frac{(2n+1)|x-y^{\prime}|}{|x-y^{\prime}|\log(|x-y^{\prime}|^{-1})} = \frac{2n+1}{\log(2^{n}(2n+1))}.
\end{align*}

In each case we have $\frac{ |\phi(x)-\phi(y)|}{\omega(|x-y|)}$ is bounded. Therefore, $\phi$ is an $\alpha$-H\"older continuous map for any $\alpha\in (0,1)$.
\end{example}

We recall   the map $\varrho$  presented in \eqref{mapavv} has (lower and upper) metric mean dimension equal to 1. In that case, for any $k\geq 1$, $\varrho$  has  a $k^{k}$-horseshoe with length  equal to $ \frac{6}{\pi^{2}k^{2}}$. We can prove that $\varrho$ is not $\alpha$-H\"older for none $\alpha\in (0,1)$, because the number of legs of each  horseshoe is so big, and therefore the slopes of the map on each subinterval increases very quickly.   Next, we will  present a path (depending of the length of the horseshoes) consisting of  continuous maps such that each of them have  less legs than $\varrho$, their metric mean dimension  is equal to one, however, they are  not $\alpha$-H\"older for none $\alpha\in (0,1)$.  
 
\begin{example}\label{1example}  Fix any  $\beta\in (1,\infty)$ and take $g\in C^{0}([0,1])$, defined by $x\mapsto |1-|3x-1||$. Set  $a_{n}=\sum_{k=1}^{n}\frac{C}{k^{\beta}}$ for $n\geq 1$, where $C=\frac{1}{\sum_{k=1}^{\infty} {k^{-\beta}}}$. For each $n\geq 2$, let 
 $T_{n}: I_{n}:=[a_{n-1},a_{n}] \rightarrow [0,1] $ be the unique increasing affine map from $I_{n}$  onto $[0,1]$.   
Consider the map  $\phi_{\beta}:[0,1]\rightarrow [0,1]$ defined by   $$\phi_{\beta}|_{I_{n}}= T_{n}^{-1}\circ g^{{n}}\circ T_{n}\quad \text{ for each }n\geq 2.$$ Hence, each $I_{n}$, whose length is $\frac{C}{n^{\beta}}$, is a $3^{n}$-horseshoe for $\phi$. From Theorem \ref{misiu} and Proposition \ref{ddf} we have $$ \underline{\emph{mdim}}_{\emph{M}}([0,1] ,|\cdot |,\phi_{\beta})= \overline{\emph{mdim}}_{\emph{M}}([0,1] ,|\cdot |,\phi_{\beta})=1.$$  
Hence, $\{\phi_{\beta}\}_{\beta\in(1,\infty)}$ is a path of continuous map with metric mean dimension equal to 1. 
Next, we will prove that $\phi$ is not $\alpha$-H\"older for none $\alpha\in (0,1)$. Note that each $I_{n}$ can be divided into $3^{n}$ sub-intervals with the same length, $I_{n,1}, I_{n,2},\dots, I_{n,3^{{n}}}$ such that for any $j\in\{1,2\dots,3^{{n}}\}$ we have
$$\phi_{\beta}(x)=\begin{cases}  3^{{n}}(x-a_{n-1})-(j-1)(a_{n}-a_{n-1})+a_{n-1} & \text{ if }x\in I_{n, j} \text{ for } j \text{ odd }\\
-3^{{n}}(x-a_{n-1})+(j-1)(a_{n}-a_{n-1})+a_{n} &   \text{ if }x\in I_{n, j} \text{ for } j \text{ even}.\end{cases}$$
Hence, for each $x,y\in I_{n,j}$, we have $$ |x-y|\leq {|I_{n,j}|}= \frac{|I_{n}|}{3^{{n}}}=  \frac{C}{3^{{n}}[n^{\beta}]}\quad\text{and}\quad  |\phi_{\beta}(x) -\phi_{\beta}(y)|=3^{{n}}|x-y|. $$
Therefore, if $x=a_{n-1}\in I_{n,1}  $ and $y=a_{n-1}+\frac{a_{n}-a_{n-1}}{3^{{n}}} \in I_{n,1}$, we have
$$ \frac{|\phi_{\beta}(x)-\phi_{\beta}(y)|}{|x-y|^{\alpha}}= 3^{{n}}|x-y|^{1-\alpha}\leq \frac{3^{{n}}C^{1-\alpha}}{3^{(1-\alpha){n}}[n^{\beta}]^{(1-\alpha)}} = \frac{3^{\alpha {n}}C^{1-\alpha}}{[n^{\beta}]^{(1-\alpha)}} \rightarrow \infty $$ as $n\rightarrow \infty$. Thus, $\phi_{\beta}$ is not $\alpha$-H\"older. 
\end{example}

In the above example, we can see for every ${\beta}$ and $\gamma$ in $(1,\infty)$, we have $\phi_{\beta}$ and $\phi_{\gamma}$ are topologically conjugate and they have the same metric mean dimension. As we mentioned in Remark
\ref{kjjj},  the metric mean dimension is not invariant under topological conjugacy. In \cite{JeoPMD}, Remark 3.2, is showed a path of continuous maps, $\{\phi_{\beta}\}_{\beta\in (0,\infty)}$, such that for every ${\beta}$ and $\gamma$ in $(0,\infty)$,   $\phi_{\beta}$ and $\phi_{\gamma}$ are topologically conjugate, however, $$ \text{mdim}_{\text{M}}([0,1],|\cdot|,\phi_{\beta})\neq \text{mdim}_{\text{M}}([0,1],|\cdot|,\phi_{\gamma})\quad\text{if }\beta\neq \gamma. $$
 
 \medskip

In the next theorem we prove the existence of H\"older continuous map on the interval which have positive metric mean dimension. Note for the map presented in Example \ref{equaltozero}, the increase of the number of legs, which is $2n+1$ for each $n\geq 1$, is very slow compared to the decrease in size of the horseshoes, which is $\frac{1}{2^{n}}$ for each $n\geq 1$.  Therefore, that map has metric mean dimension equal to zero, from Theorem \ref{misiu} and Proposition \ref{ddf}. If we keep the same subintervals $I_{n}=[2^{-n},2^{-n+1}]$ and we wish to obtain a continuous map on $[0,1]$ with positive metric mean dimension,  we must have $\underset{n\rightarrow \infty}{\limsup}\frac{\log 2^{n}}{\log s_{n}}<\infty$, where $s_{n}$ is the number of legs of the map on each $I_{n}$. This fact implies that $s_{n}$ must increase very quickly and therefore, the H\"older exponent of the map must be zero or very small. Hence,  we must modify the size of the horseshoe in order to obtain a continuous map with positive metric mean dimension and   a greater H\"older exponent: the number of legs must not increase so fast (otherwise the H\"older exponent decreases) and the size of the horseshoe must not decay too fast (otherwise the metric mean dimension decreases). In the next theorem, is presented an example where the size of the horseshoe is $\frac{C}{3^{nr}}$, for a fixed $r\in(0,\infty)$ and a constant $C>0$, and the number of legs is $3^{n}$, for any $n\in\mathbb{N}$.  These sequences, $\frac{C}{3^{nr}}$ and $3^{n}$, are reasonably similar and we will prove that in this case the map has positive metric mean dimension and a considerable H\"older exponent.

 \begin{theorem}\label{example-a}  Fix $ a\in [0,1)$ and take  $\alpha= 1-a$. There exists an $\alpha$-H\"older continuous map $\phi_{a}:[0,1]\rightarrow [0,1]$   such that $$   \underline{\emph{mdim}}_{\emph{M}}([0,1] ,| \cdot |,\phi_{a})=\overline{\emph{mdim}}_{\emph{M}}([0,1] ,| \cdot |,\phi_{a})=a.   $$
\end{theorem}
\begin{proof} If $a=0$ we can take $\phi_{0}$ as the identity on $[0,1]$. 
Fix  $r>0$ and let $ a=\frac{1}{r+1}$.  Set $a_{0}=0$ and $a_{n}= \sum_{i=0}^{n-1}\frac{C}{3^{ir}}$ for $n\geq 1$, where $C=\frac{1}{\sum_{i=0}^{\infty}\frac{1}{3^{ir}}}= \frac{3^{r}-1}{3^{r}}$. For each $n\geq 1$, set $I_{n}=[a_{n-1},a_{n}]$ and take 
 $T_{n}:I_{n}\rightarrow [0,1]$ and $g$ as in Example \ref{1example}.    Set $\phi_{a}:[0,1]\rightarrow [0,1]$, given by   $\phi_{\alpha}|_{I_{n}}= T_{n}^{-1}\circ g^{n}\circ T_{n}$ for any $n\geq 1$.   It follows from Theorem \ref{misiu} that  
$$  \underline{\text{mdim}}_{\text{M}}([0,1], | \cdot |,\phi_{a})=  \overline{\text{mdim}}_{\text{M}}([0,1],| \cdot |,\phi_{a})=a .$$

We will prove that  $\phi_{a}: [0,1] \to [0,1]$ is $\alpha$-H\"older with  $\alpha = \frac{r}{r+1}=1-a.$  
Let $x,y\in [0,1]$ be two distinct points.  Thus, there exist $n\geq 1$ and $m\geq 1$ such that $x\in I_{n}$ and $y\in I_{m}$. Given that each $I_{k}$ is $\phi_{a}$-invariant, we have $\phi_{a}(x)\in I_{n}$ and $\phi_{a}(y)\in I_{m}$. We have the following cases:

\medskip

\noindent \textbf{Case} $x=0:$ Since $\phi_{a}(0)=0$, we have  $$|\phi_{a}(0)-\phi_{a}(y)|=|\phi_{a}(y)|\leq \frac{1-3^{-mr}}{1-3^{-r}} .$$ Furthermore,  $$ \frac{1-3^{(-m+1)r}}{1-3^{-r}} \leq|y|\leq \frac{1-3^{-mr}}{1-3^{-r}}\quad\text{ and thus }\quad \frac{1}{|y|}\leq \frac{1}{\frac{1-3^{(-m+1)r}}{1-3^{-r}}}\quad\text{and}\quad \frac{1}{\log\frac{1}{|y|}}\leq \frac{1}{\log \frac{1-3^{-mr}}{1-3^{-r}}}. $$ Hence, if  $\omega $  is the map   defined in Lemma \ref{hfnff}, we have $$ \frac{ |\phi(y)|}{\omega(|y|)}<\frac{\frac{1-3^{-mr}}{1-3^{-r}}}{\frac{1-3^{(-m+1)r}}{1-3^{-r}} \log \frac{1-3^{-mr}}{1-3^{-r}} } = \frac{1-3^{-mr}}{[1-3^{(-m+1)r}]\log \frac{1-3^{-mr}}{1-3^{-r}} }   ,$$
which converges to $\frac{1}{\log \frac{1}{ 1-3^{-r}} }$. This fact implies that $\frac{ |\phi(y)|}{\omega(|y|)}$ is bounded.

\medskip 

\noindent  \textbf{Case}  $n=m+k$, \textbf{with} $k>1$ \textbf{for a fixed} $m$:   we know that 

\[
|\phi_{a}(x)-\phi_{a}(y)|\leq \frac{3^{(1-m)r}-3^{-nr}}{1-3^{-r}}=\frac{3^{(n-m+1)r}-1}{3^{nr}(1-3^{-r})}\quad\text{ and }\quad
 \frac{1}{|x-y|}\leq \frac{(1-3^{-r})3^{nr}}{3^{(n-m)r}-1},\]
hence
\begin{equation}\label{fbbg}
\frac{|\phi_{a}(x)-\phi_{a}(y)|}{|x-y|^{\alpha}}\leq \frac{3^{nr\alpha} (1-3^{-r})^{\alpha-1} [3^{(n-m+1)r}-1]}{3^{nr}[3^{(n-m)r}-1]^{\alpha}} ={\frac{3^{nr(\alpha-1)} (1-3^{-r})^{\alpha-1} [3^{(k+1)r}-1]}{(3^{kr}-1)^{\alpha}}}.
\end{equation}

\medskip \noindent Note that, if the sequence $\{k\}$ is bounded, then \eqref{fbbg} is convergent to $0$ as $n\to \infty$, since $\alpha-1<0$. Hence   \eqref{fbbg} is bounded if $\{k\}$ is bounded. Therefore, we can assume that the sequence $\{k\}$ is unbounded. From \eqref{fbbg}, we write     

\begin{align*}
\frac{|\phi_{a}(x)-\phi_{a}(y)|}{|x-y|^{\alpha}} &\leq 3^{nr(\alpha-1)} (1-3^{-r})^{\alpha-1}\  \frac{\frac{3^{(k+1)r}-1}{3^{kr\alpha}}}{\frac{(3^{kr}-1)^\alpha}{3^{kr\alpha}}}=3^{nr(\alpha-1)} (1-3^{-r})^{\alpha-1}  \frac{(3^{kr+r-kr\alpha}-3^{-kr\alpha})}{\biggl(1-\frac{1}{3^{kr}}\biggr)^\alpha}.
\end{align*}
Since
$\underset{k\to \infty}{\lim}\biggl(1-\frac{1}{3^{kr}}\biggr)=1$ and $ \underset{k\to \infty}{\lim} 3^{-kr\alpha}=0,$
we have
\begin{align*}
\lim_{k\to \infty} 3^{nr(\alpha-1)} (1-3^{-r})^{\alpha-1} \ \frac{(3^{kr+r-kr\alpha}-3^{-kr\alpha})}{\biggl(1-\frac{1}{3^{kr}}\biggr)^\alpha}&=(1-3^{-r})^{\alpha-1} \lim_{k\to \infty} 3^{nr\alpha-nr+kr+r-kr\alpha}\\
&=(1-3^{-r})^{\alpha-1} \  3^r\ \lim_{k\to \infty}3^{(n-k)r\alpha-(n-k)r}\\
&=(1-3^{-r})^{\alpha-1} \  3^r\ \lim_{k\to \infty} 3^{(n-k)(\alpha-1)r}\\
&=(1-3^{-r})^{\alpha-1} \  3^r\ \lim_{k\to \infty} 3^{m(\alpha-1)r}\\
&= (1-3^{-r})^{\alpha-1} 3^{m(\alpha-1)r+r}.
\end{align*}
Thus, in any case,  $\frac{|\phi_{a}(x)-\phi_{a}(y)|}{|x-y|^{\alpha}}$ is bounded with $x\in I_n$, $y\in I_m$, $n>m+1$.

\medskip \noindent 
 \textbf{Case} $n=m+1$:  we have $I_m=[a_{m-1}, a_m] ,$ $I_{m+1}=[a_m, a_{m+1}],$
$$I_m=I^{1}_m\cup \cdots \cup I^{3^m}_{m}; \ \ I_{m+1}=I^{1}_{m+1}\cup \cdots \cup I^{3^{m+1}}_{m+1} \ \  \text{and} \ \ I_m\cap I_{m+1}=\{a_{m}\},$$ where $|I_{m}^{j}|=\frac{|I_{m}|}{3^{m}}$ and $|I_{m+1}^{j}|=\frac{|I_{m+1}|}{3^{m+1}}$. 
Suppose that $x\in I^{1}_{m+1}$ and $y\in I^{3^m}_{m}$. 
Hence \begin{equation}\label{bcmvns}
|x-y|\leq a_m+\frac{|I_{m+1}|}{3^{m+1}}-\biggl(a_m-\frac{|I_{m}|}{3^{m}}\biggr)=\frac{|I_{m+1}|}{3^{m+1}}+\frac{|I_{m}|}{3^{m}}=\frac{C(1+3^{r+1})}{3^{m(r+1) +1}}.
\end{equation}
It is easy to check that 
$$\phi _{a}(y)=3^m(y-a_m)+a_m\quad\text{and}\quad \phi _{a} (x)=3^{m+1}(x-a_m)+a_m.$$
Hence, 
\begin{align*}
 |\phi_{a}(x)-\phi_{a}(y)|&=|3^{m+1}(x-a_m)-3^{m}(y-a_m)|=3^{m+1}x-3^{m}y-(3^{m+1}-3^m)a_m\\
(y\leq a_{m} )\quad  & \leq  3^{m+1}x-3^{m}y-(3^{m+1}-3^m)y= 3^{m+1}(x-y).
\end{align*}
Therefore, from \eqref{bcmvns} we have
\begin{align*}
\frac{|\phi_{a}(x)-\phi_{a}(y)|}{|x-y|^{\alpha}}&\leq 3^{m+1} |x-y|^{1-\alpha}\leq 3^{m+1}\biggl(\frac{C(1+3^{r+1})}{3^{m(r+1) +1}} \biggr)^{1-\alpha}= \frac{3^{m+1}(C(1+3^{r+1}))^{1-\alpha}}{3^{(m(r+1)+1)(1-\alpha)}}\\
&= 3^{m+1 +(m(r+1)+1)(\alpha-1)}(C(1+3^{r+1}))^{1-\alpha}=3^{m(\alpha(r+1)-r)}3^{\alpha}(C(1+3^{r+1}))^{1-\alpha},
\end{align*}
which is bounded if and only if 
$\alpha(r+1)-r\leq 0,$  that is, if $\alpha\leq \frac{r}{1+r}.$   

If $x\in I^{j}_{m+1}$ and $y\in I^{i}_{m}$ for some $j$ and $i$, then there are ${x}'\in I^{1}_{m+1}$ and ${y}'\in I^{3^m}_{m}$ such that  $ \phi_a(x)= \phi_a({x}') $ and $ \phi(y)= \phi({y}').$ Moreover, $|x-y|\geq |{x}'-{y}'|$, which provides 
\begin{align*}
| \phi_a(x)-\phi(y)| &=| \phi_a({x}')- \phi_a({y}')|\leq 3^{m(\alpha(r+1)-r)}3^{\alpha}(C(1+3^{r+1}))^{1-\alpha}|{x}'-{y}'|^{\alpha} \\
&\leq  3^{m(\alpha(r+1)-r)}3^{\alpha}(C(1+3^{r+1}))^{1-\alpha}|{x}-{y}|^{\alpha}.
\end{align*}

\noindent \textbf{Case} $n=m$: we have $\phi (x)=3^n(x-a_n)+a_n$ for any $x\in I^{3n}_{n}$. Hence, if  $ x, y\in I^{3n}_{n}$, then
$
|\phi_{a}(x)-\phi_{a}(y)| =3^n |x-y|$ and therefore 
\begin{align*}
\frac{|\phi_{a}(x)-\phi_{a}(y)|}{|x-y|^{\alpha}}&= 3^n |x-y|^{1-\alpha}\leq  3^n\ |I^{3n}_{n}|^{1-\alpha}=3^n\biggl(\frac{|I_n|}{3^n}\biggr)^{1-\alpha}=3^n \biggl(\frac{C}{3^{(n-1)r}3^n}\biggr)^{1-\alpha}\\
&=C^{1-\alpha}3^{n-(1-\alpha)(nr+n)}=C^{1-\alpha}3^{r-\alpha r} 3^{n( \alpha -r+\alpha r)},
\end{align*}
which is bounded if and only if 
$ \alpha -r+\alpha r\leq 0$, that is, if and only if $  \alpha \leq \frac{r}{1+r} .$

 We can conclude that $\phi_a$ is an $(1-a)$-H\"older continuous map.   \end{proof}

Let $X$ be a compact metric space with metric $d$.  If $\phi_{1},\dots,\phi_{n}\in C^{0}(X)$ and $$\overline{\text{mdim}}_{\text{M}}(X ,d,\phi_{k})=\underline{\text{mdim}}_{\text{M}}(X ,d,\phi_{k})=a_{k},\quad\text{for any }k=1,\dots,n,$$ 
we have \begin{equation*}\label{jme}  \underline{\text{mdim}}_{\text{M}}(X\times \cdots\times X ,d^{n},\phi_{1}\times\cdots \times \phi_{n})  =  \overline{\text{mdim}}_{\text{M}}(X\times \cdots\times X ,d^{n},\phi_{1}\times\cdots \times \phi_{n})=   \sum_{k=1}^{n}a_{k},\end{equation*} 
where  \begin{equation*}\label{bnm}d^{n}((x_{1},\dots,x_{n}),(y_{1},\dots,y_{n}))=d(x_{1},y_{1}) +\cdots+ d (x_{n},y_{n}),\quad \text{ for  }x_{1},\dots,x_{n},y_{1},\dots,y_{n}\in X.\end{equation*} (see \cite{JeoPMD}, Theorem 3.13, v). Hence, it  follows from Theorem  \ref{example-a} that: 

\begin{corollary}
Fix $b\in [0,n]$ and $\alpha ={1-\frac{b}{n}}$. There exists $\psi_{b}\in H^{\alpha} ([0,1]^{n})$ such that $$\overline{\emph{mdim}}_{\emph{M}}([0,1]^{n}  ,d^{n},\psi_{b})=\underline{\emph{mdim}}_{\emph{M}}([0,1]^{n}  ,d^{n},\psi_{b})=b.$$  
\end{corollary} 
\begin{proof}
Take $\phi_{a}:[0,1]\rightarrow [0,1]$ defined in the proof of Theorem  \ref{example-a}, where $a=\frac{b}{n}$, and set  $\psi_{b}=\phi_{a}\times \cdots \times \phi_{a}:[0,1]^{n}\rightarrow[0,1]^{n}$. We have $$ \underline{\text{mdim}}_{\text{M}}([0,1]^{n} ,d^{n},\psi_{b})  =  \overline{\text{mdim}}_{\text{M}}([0,1]^{n} ,d^{n},\psi_{b})=  na=b.$$
 Given that each $\phi_{a}$ is an $(1-\frac{b}{n})$-H\"older continuous map, we have $\psi_{b}$ is an $(1-\frac{b}{n})$-H\"older continuous map. 
\end{proof}

\section{Some conjectures about this research}

Note in Theorem \ref{example-a}  there exists a relationship between the H\"older  exponent $\alpha$ and the metric mean dimension of a continuous map on the interval. We will present the next conjectures about this relationship which can be the subject of future research.

\begin{conjecturea}\label{conj1} If $\phi:[0,1]\rightarrow [0,1]$ is an $\alpha$-H\"older continuous map, then    $$\emph{mdim}_{\emph{M}}([0,1],|\cdot |,\phi) \leq 1-\alpha.$$\end{conjecturea}
 
 Note in Theorem \ref{example-a} we prove there exists an $\alpha$-H\"older continuous map $\phi:[0,1]\rightarrow [0,1]$ with $\text{mdim}_{\text{M}}([0,1],|\cdot|,\phi)= 1-\alpha $, where $\alpha\in (0,1)$.  From \eqref{boundd} we have does not exist any continuous map on the interval with metric mean dimension biggest to 1.

\begin{conjectureb} There is no any  $\alpha$-H\"older continuous map  $\phi:[0,1]\rightarrow [0,1]$,  with  $\alpha>0$ and  $$\emph{mdim}_{\emph{M}}([0,1],|\cdot |,\phi) = 1.$$\end{conjectureb}

 Note if  Conjecture A is true, then Conjecture B is a consequence of Conjecture A.

\begin{conjecturec}If $\phi:[0,1]\rightarrow [0,1]$ is an $\alpha$-H\"older continuous map for any  $\alpha\in (0,1)$, then   $$\emph{mdim}_{\emph{M}}([0,1],|\cdot |,\phi) = 0.$$\end{conjecturec}

Note if  Conjecture A is true, then Conjecture C is a consequence of Conjecture A.

\end{document}